\documentclass[reqno,a4paper]{amsart}
\usepackage{amssymb}
\usepackage{amsmath}
\begin{document} 
\newtheorem{prop}{Proposition}[section]
\newtheorem{Def}{Definition}[section] \newtheorem{theorem}{Theorem}[section]
\newtheorem{lemma}{Lemma}[section] \newtheorem{Cor}{Corollary}[section]

\title[Maxwell-Chern-Simons-Higgs in Lorenz gauge]{\bf Local solutions with infinite energy of the Maxwell-Chern-Simons-Higgs system in Lorenz gauge}
\author[Hartmut Pecher]{
{\bf Hartmut Pecher}\\
Fachbereich Mathematik und Naturwissenschaften\\
Bergische Universit\"at Wuppertal\\
Gau{\ss}str.  20\\
42119 Wuppertal\\
Germany\\
e-mail {\tt pecher@math.uni-wuppertal.de}}
\date{}

\begin{abstract}
We consider the Maxwell-Chern-Simons-Higgs system in Lorenz gauge and use a null condition to show local well-psoedness for low regularity data. This improves a recent result of J. Yuan.
\end{abstract}
\maketitle
\renewcommand{\thefootnote}{\fnsymbol{footnote}}
\footnotetext{\hspace{-1.5em}{\it 2010 Mathematics Subject Classification:} 
35Q40, 35L70 }
\normalsize 
\setcounter{section}{0}
\section{Introduction and main results}
\noindent The Lagrangian of the (2+1)-dimensional Maxwell-Chern-Simons-Higgs model which was proposed in \cite{LLM} is given by
\begin{align*}
{\mathcal L}  = &-\frac{1}{4} F^{\mu \nu} F_{\mu \nu} + \frac{\kappa}{4} \epsilon^{\mu \nu \rho} F_{\mu \nu} A_{\rho} + D_{\mu}\phi\overline{D^{\mu} \phi} \\
& + \frac{1}{2} \partial_{\mu}N \partial^{\mu} N - \frac{1}{2}(e|\phi|^2+\kappa N - ev^2)^2 - e^2 N^2 |\phi|^2 
\end{align*}
in Minkowski space ${\mathbb R}^{1+2} = {\mathbb R}_t \times {\mathbb R}_x^2$ with metric $g_{\mu \nu} = diag(1,-1,-1)$. We use the convention that repeated upper and lower indices are summed, Greek indices run over 0,1,2 and Latin indices over 1,2. Here 
\begin{align*}
D_{\mu}  & := \partial_{\mu} - ieA_{\mu} \\
 F_{\mu \nu} & := \partial_{\mu} A_{\nu} - \partial_{\nu} A_{\mu} \\
\end{align*}
Here $F_{\mu \nu} : {\mathbb R}^{1+2} \to {\mathbb R}$ denotes the curvature, $\phi : {\mathbb R}^{1+2} \to {\mathbb C}$ and $N: {\mathbb R}^{1+2} \to {\mathbb R}$ are scalar fields, and $A_{\mu} : {\mathbb R}^{1+2} \to {\mathbb R}$ are the gauge potentials. $e$ is the charge of the electron and $\kappa > 0$ the Chern-Simons constant, $v$ is a real constant. We use the notation $\partial_{\mu} = \frac{\partial}{\partial x_{\mu}}$, where we write $(x^0,x^1,...,x^n) = (t,x^1,...,x^n)$ and also $\partial_0 = \partial_t$. $\epsilon^{\mu \nu \rho}$ is the totally skew-symmetric tensor with $\epsilon^{012} = 1$.

The corresponding Euler-Lagrange equations are given by
\begin{align}
\label{2.1}
& \partial_{\lambda} F^{\lambda \rho} + \frac{\kappa}{2} \epsilon^{\mu \nu \rho} F_{\mu \nu} + 2e Im( \phi  \overline{D^{\rho} \phi})  = 0\\
\label{2.2}
&D_{\mu} D^{\mu} \phi  + U_{\overline{\phi}}(|\phi|^2,N) = 0 \\
\label{2.3}
& \partial_{\mu} \partial^{\mu} N + U_N (|\phi|^2,N) = 0 \, ,
\end{align}
where
\begin{align*}
U_{\overline{\phi}}(|\phi|^2,N) &= (e|\phi|^2+\kappa N -ev^2)\phi + e^2 N^2 \phi \\
U_N(|\phi|^2,N) & = \kappa(e|\phi|^2 + \kappa N -ev^2) + 2e^2N |\phi|^2 \, .
\end{align*}
(\ref{2.1}) can be written as follows
\begin{align}
\label{1.1*}
&-\Delta A_0 + \partial_t(\partial_1 A_1 +\partial_2 A_2) + \kappa F_{12} + 2e Im(\phi \overline{D^0 \phi})  = 0 \\
\label{1.2*}
&(\partial_t^2-\partial_2^2)A_1 - \partial_1(\partial_t A_0 -\partial_2 A_2) - \kappa F_{02} + 2e Im(\phi \overline{D^1 \phi}) = 0 \\
\label{1.3*}
&(\partial_t^2-\partial_1^2)A_2 - \partial_2(\partial_t A_0 -\partial_1 A_1) + \kappa F_{01} + 2e Im(\phi \overline{D^2 \phi}) = 0 \, .
\end{align}
The initial conditions are
\begin{align}
\label{IC}
&A_{\nu}(0) = a_{\nu 0} \quad , \quad (\partial_t A_{\nu})(0) = a_{\nu 1}  \quad , \quad \phi(0) = \phi_0 \quad , \quad (\partial_t \phi)(0) = \phi_1 \\ 
\nonumber
&N(0) = N_0 \quad , \quad (\partial_t N)(0) = N_1   \, .
\end{align}
The Gauss law constraint (\ref{1.1*}) requires the initial data to fulfill the following condition:
\begin{equation}
\Delta a_{00} - \partial_1 a_{11} - \partial_2 a_{21} -\kappa(\partial_1 a_{20} - \partial_2 a_{10}) - 2e Im(\phi_0 \overline{\phi}_1) +2e^2 a_{00}|\phi_0|^2 = 0 \,.
\end{equation} 
The energy $E(t)$ of (\ref{2.1}),(\ref{2.2}),(\ref{2.3}) is (formally) conserved, where
\begin{align*}
E(t) = &\int \frac{1}{2} \sum_i F_{0i}^2(x,t) + \frac{1}{2} F_{12}^2(x,t) \\&+  \sum_{\mu} |D_{\mu} \phi(x,t)|^2 + \sum_{\mu} |\partial_{\mu} N(x,t)|^2 + U(|\phi|^2,N)(x,t) dx
\end{align*}
with
$$ U(|\phi|^2,N) = \frac{1}{2}(e|\phi|^2+\kappa N -ev^2)^2 +e^2 N^2 |\phi|^2 \, . $$
There are two possible natural asymptotic conditions to make the energy finite: either the "nontopological" boundary condition
$ (\phi,N,A) \rightarrow (0,\frac{ev^2}{\kappa},0)$ as $|x| \to \infty $
or the "topological" boundary condition
$(|\phi|^2,N,A) \rightarrow (v^2,0,0)$ as $|x| \to \infty \, . $

We decide to study the "nontopological" boundary condition. Replacing $N$ by $N-\frac{ev^2}{\kappa}$ and denoting it again by $N$  we obtain $(\phi,N,A) \to (0,0,0)$ as $|x| \to \infty$, thus leading to solutions in standard Sobolev spaces, and in (\ref{2.2}),(\ref{2.3}) we now have
\begin{align}
\label{*}
U_{\overline{\phi}}(|\phi|^2,N) &= (e|\phi|^2+\kappa N)\phi + e^2(N+\frac{ev^2}{\kappa})^2 \phi \\
\label{****}
 U_N(|\phi|^2,N) &= \kappa (e|\phi|^2 + \kappa N) +2e^2(N+\frac{ev^2}{\kappa}) |\phi|^2 
\end{align}

For the "topological" boundary condition the problem can also be reduced to a system for $(\phi,N,A)$ which fulfills $(\phi,N,A) \to (0,0,0)$ as $|x| \to \infty$ for a modified function $\phi$, if one makes the assumption that $\phi \to \lambda \in{\mathbb C}$ as $|x| \to \infty$ with $|\lambda| = v$. In this case one simply replaces $\phi$ by $\phi - \lambda$. For details we refer to Yuan's paper \cite{Y}. It is easy to see that the system in this case can be studied in the same way as in the "nontopological" case.

The equations (\ref{2.1}),(\ref{2.2}),(\ref{2.3}) are invariant under the gauge transformations
$$ A_{\mu} \rightarrow A'_{\mu} = A_{\mu} + \partial_{\mu} \chi \, , \, \phi \rightarrow \phi' = \exp(ie\chi) \phi \, , \, D_{\mu} \rightarrow D'_{\mu} = \partial_{\mu}-ieA'_{\mu} \, . $$
We consider exclusively the Lorenz gauge
$$ \partial^{\mu} A_{\mu} = 0 \, , $$
so that we have to assume that the data fulfill
\begin{equation}
\label{**}
 \partial^{\mu} a_{\mu} = 0 \, . 
 \end{equation}
We want to prove local well-posedness of the Cauchy problem for (\ref{2.1}),(\ref{2.2}),(\ref{2.3}) for data with minimal regularity assumptions especially for $\phi_0$ and $\phi_1$.

Chae-Chae \cite{CC} assumed $(\phi_0,\phi_1) \in H^2 \times H^1$ and proved local and even global well-posedness using energy conservation. This was improved by J. Yuan \cite{Y} to $(\phi_0,\phi_1),(a_{\mu 0},a_{\mu 1}),(N_0,N_1) \in H^s \times H^{s-1}$ with $s > \frac{3}{4}$ , who obtained a local solution $\phi,A_{\mu},N \in C^0([0,T],H^s) \cap C^1([0,T],H^{s-1})$, which is unique in a suitable subset of $X^{s,b}$-type. Using energy conservation this solution exists globally, if $s\ge 1$.

We now further lower down the regularity of the data to $(\phi_0,\phi_1) \in H^s \times H^{s-1}$,  $(a_{\mu 0},a_{\mu 1}) \in H^{2s-\frac{3}{4}-} \times H^{2s-\frac{7}{4}-}$ , $(N_0,N_1) \in H^{\frac{1}{2}} \times H^{-\frac{1}{2}}$ under the condition $ s > \frac{1}{2} $ . We obtain a local solution $\phi \in C^0([0,T],H^s) \cap C^1([0,T],H^{s-1})$, $A_{\mu} \in C^0([0,T],H^{2s-\frac{3}{4}-}) \cap C^1([0,T],H^{2s-\frac{7}{4}-}) $ , $N\in C^0([0,T],H^{\frac{1}{2}}) \cap C^1([0,T],H^{-\frac{1}{2}-})$,  which is unique in a suitable subspace of $X^{s,b}$-type.

Moreover it is easy to see that as a consequence of these results unconditional uniqueness holds for the solutions of Yuan, i.e. , for $\phi,A_{\mu},N \in C^0([0,T],H^s) \cap C^1([0,T],H^{s-1})$ and $s> \frac{3}{4}$ , especially global well-posedness for finite energy solutions (s=1).

Whereas Chae-Chae only used standard energy type estimates Yuan applied bilinear Strichartz type estimates which were given in the paper of d'Ancona, Foschi and Selberg \cite{AFS}. We also use this type of estimates but additionally take advantage of a crucial null condition of the term $A_{\mu} \partial^{\mu} \phi$ in the wave eqution for $\phi$. This was detected by Klainerman-Machedon \cite{KM} and Selberg-Tesfahun \cite{ST} for the Maxwell-Klein-Gordon equations and also by Selberg-Tesfahun \cite{ST1} for the corresponding problem for the Chern-Simons-Higgs equations. When combined with the bilinear Strichartz type estimates this leads to the improved lower bound  for the regularity for the problem at hand.

We denote the Fourier transform $\mathcal{F}$ with respect to space as well as to space and time by $\,\widehat{}$ . The operator
$\langle \nabla \rangle^{\alpha}$ is defined by $\mathcal{F}(\langle \nabla \rangle^{\alpha} f)(\xi) = \langle \xi \rangle^{\alpha} (\widehat{f})(\xi)$, where $\langle \cdot \rangle := (1+|\cdot|^2)^{\frac{1}{2}}$ . Define
$a+ := a+\epsilon$ for $\epsilon > 0$ sufficiently small, so that $a < a+ < a++$ and similarly $a-- < a- < a$, and $a-b\,-$ means $(a-b)-$.\\
Besides the Sobolev spaces $H^{k,p}$ we use the spaces $X^{s,b}_{\pm}$ of Bougain-Klainerman-Machedon type defined as the completion of ${\mathcal S}({\mathbb R}^3)$ with respect to the norm $\|u\|_{X^{s,b}_{\pm}} = \| \langle \xi \rangle^s (\tau\pm |\xi| \rangle^b \widehat{u}(\tau,\xi)\|_{L^2_{\tau \xi}}$ and similarly the wave-Sobolev spaces $X^{s,b}$ with norm $\|u\|_{X^{s,b}} = \|\langle \xi \rangle^s \langle |\tau| - |\xi| \rangle^b \widehat{u}(\tau,\xi)\|_{L^2_{\tau \xi}}$ . The spaces of restrictions to $[0,T] \times {\mathbb R}^2$ with induced norms are $X^{s,b}_{\pm}[0,T]$ and $X^{s,b}[0,T]$. Remark that $\|u\|_{X^{s,b}_{\pm}} \le \|u\|_{X^{s,b}}$ for $b \le 0$ and the reverse estimate for $b \ge 0$.

We now formulate our main results. One easily checks that a solution of (\ref{2.1}),(\ref{2.2}),(\ref{2.3}) (with (\ref{*}),(\ref{****})) under the Lorenz condition
\begin{equation}
\label{1.4*}
\partial^{\mu} A_{\mu} = 0
\end{equation}
also fulfills the following system
\begin{align}
\label{1.1}
&(\square +1)A_0 = -\kappa F_{12} -2e Im(\phi \overline{D_0 \phi}) +A_0 \\
\label{1.2}
&(\square +1)A_i = -\kappa \epsilon^{ij} F_{0j} -2e Im(\phi \overline{D_i \phi}) +A_i \\
\label{1.3}
&(\square +1)\phi =  2ie A_0 \partial_0 \phi -2ie A^j \partial_j \phi -e^2 A^j A_j \phi +e^2 A_0^2 \phi - U_{\overline{\phi}}(|\phi|^2,N) + \phi \\
\label{1.4}
&(\square +1)N = - U_N(|\phi|^2,N) + N \, .
\end{align}
Here we replaced $\square$ by $\square +1$ by adding a linear terms on both sides of the equations in oder to avoid the operator $(-\Delta)^{-\frac{1}{2}}$, which is unpleasent especially  in two dimensions.

Defining
\begin{align*}
&A_{\mu,\pm} = \frac{1}{2}(A_{\mu} \pm i^{-1} \langle \nabla \rangle^{-1} \partial_t A_{\mu}) &\Leftrightarrow &A_{\mu}=A_{\mu,+} + A_{\mu,-} \,,\, \partial_t A_{\mu}=i\langle \nabla \rangle(A_{\mu,+}-A_{\mu,-}) \\
&\phi_{\pm} = \frac{1}{2}(\phi \pm i^{-1} \langle \nabla \rangle^{-1} \partial_t \phi) &\Leftrightarrow &\phi = \phi_+ + \phi_- \, , \, \partial_t \phi = i \langle \nabla \rangle(\phi_+ - \phi_-) \\
&N_{\pm} = \frac{1}{2}(N \pm i^{-1} \langle \nabla \rangle^{-1} \partial_t N) &\Leftrightarrow &N = N_+ + N_- \, , \, \partial_t N= i \langle \nabla \rangle(N_+ - N_-)
\end{align*}
we obtain the equivalent system
\begin{align}
\label{1.1'}
(i \partial_t \pm \langle \nabla \rangle) A_{0,\pm} & = \pm 2^{-1} \langle \nabla \rangle^{-1} ( \, {\mbox R.H.S.\, of} \,(\ref{1.1})) \\
\label{1.2'}
(i \partial_t \pm \langle \nabla \rangle) A_{j,\pm} & = \pm 2^{-1} \langle \nabla \rangle^{-1} ( \, {\mbox R.H.S.\, of} \,(\ref{1.2})) \\
\label{1.3'}
(i \partial_t \pm \langle \nabla \rangle) \phi_{\pm} & = \pm 2^{-1} \langle \nabla \rangle^{-1} ( \, {\mbox R.H.S.\, of} \,(\ref{1.3})) \\
\label{1.4'}
(i \partial_t \pm \langle \nabla \rangle) N_{\pm} & = \pm 2^{-1} \langle \nabla \rangle^{-1} ( \, {\mbox R.H.S.\, of} \,(\ref{1.4})) 
\end{align}

We obtain the following result:                                  
\begin{theorem}
\label{Theorem1}
Assume $ 1 > s > \frac{1}{2}$ and 
\begin{align*}
 &\phi_0 \in H^s \, , \, \phi_1 \in H^{s-1} \, , \, a_{\mu 0} \in H^{2s-\frac{3}{4}-} \, , \, a_{\mu 1} \in H^{2s-\frac{7}{4}-}  \, (\mu =0,1,2) \, ,\\ &n_0 \in H^{\frac{1}{2}} \, , \, n_1 \in H^{-\frac{1}{2}} \, . 
\end{align*}
There exists  $T>0$ such that
the system (\ref{1.1}),(\ref{1.2}),(\ref{1.3}),(\ref{1.4}) with (\ref{*}),(\ref{****}) and Cauchy conditions
$$A_{\mu}(0) = a_{\mu 0} \, , \, \partial_tA_{\mu}(0) = a_{\mu 1} \, , \, \phi_0 = \phi_0 \, , \, \partial_t \phi(0) = \phi_1 \, , \, N(0)=N_0 \, , \, \partial_t N(0) = N_1 $$
has a unique local solution 
\begin{align*}
& \phi \in X_+^{s,\frac{1}{2}+}[0,T] + X_-^{s,\frac{1}{2}+}[0,T] \\& A_{\mu} \in X_+^{2s-\frac{3}{4}-,2s-\frac{1}{4}--}[0,T] +X_-^{2s-\frac{3}{4}-,2s-\frac{1}{4}--}[0,T] \quad (\,{\mbox for}\, \frac{1}{2} < s \le \frac{5}{8}\,) \\
& A_{\mu} \in X_+^{2s-\frac{3}{4}-,1-}[0,T] +X_-^{2s-\frac{3}{4}-,1-}[0,T] \quad({\mbox for}\, \frac{5}{8} < s \le 1 \,) \\ 
&N \in X^{\frac{1}{2},\frac{1}{2}+}_+[0,T] + X_-^{\frac{1}{2},\frac{1}{2}+}[0,T] \, .
\end{align*} 
It has the properties
\begin{align*}
& \phi \in C^0([0,T],H^s(\mathbb{R}^2)) \cap C^1([0,T],H^{s-1}(\mathbb{R}^2)) \\
& A_{\mu} \in C^0([0,T],H^{2s-\frac{3}{4}-}(\mathbb{R}^2)) \cap C^1([0,T],H^{2s-\frac{7}{4}-}(\mathbb{R}^2)) \\
& N \in C^0([0,T],H^{\frac{1}{2}}(\mathbb{R}^2)) \cap C^1([0,T],H^{-\frac{1}{2}}(\mathbb{R}^2)) \, .
\end{align*}
\end{theorem}
This result is proven in section 2. \\
{\bf Remark:} The system (\ref{1.1})-(\ref{1.4}) is not scaling invariant, but ignoring lower order terms we can write it schematically as
\begin{align*}
\Box{A} & = \phi \nabla \phi + A \phi^2 \\
\Box{\phi} & = A \nabla \phi +A^2 \phi + N \phi^2 + \phi^3 \\
\Box{N} & = N \phi^2 \, .
\end{align*}
This system is invariant under the scaling 
$$ A(x,t) \longmapsto \lambda A(\lambda x,\lambda t) \, , \, \phi(x,t) \longmapsto \lambda \phi(\lambda x,\lambda t) \, , \,N(x,t) \longmapsto \lambda N(\lambda x,\lambda t) \, .$$ 
Thus the critical data space with respect to scaling for $A(0)$ , $\phi(0)$ , $N(0)$ (in dimension 2) is $L^2$, so that there remains a gap between this space and our minimal assumptions in Theorem \ref{Theorem1}, namely $\phi(0) \in H^{\frac{1}{2}+}$ , $A(0) \in H^{\frac{1}{4}+}$ , $N(0) \in H^{\frac{1}{2}}$.

In section 3 we show the following theorem and its corollary as a consequence of Theorem \ref{Theorem1}.
 \begin{theorem}
\label{Theorem2}
Assume $ 1 > s > \frac{1}{2}$. Moreover assume that the data satisfy the assumptions of Theorem \ref{Theorem1} and also
\begin{equation}
\label{GL}
\Delta a_{00} - \partial_1 a_{11} - \partial_2 a_{21} -\kappa(\partial_1 a_{20} - \partial_2 a_{10}) - 2e Im(\phi_0 \overline{\phi}_1) +2e^2 a_{00}|\phi_0|^2 = 0 
\end{equation} 
and
\begin{equation}
\label{21}
 \partial^{\mu} a_{\mu} = 0 \, . 
 \end{equation}
The solution of Theorem \ref{Theorem1}  is the unique solution of the Cauchy problem
for the system
\begin{align}
\label{22}
& \partial_{\lambda} F^{\lambda \rho} + \frac{\kappa}{2} \epsilon^{\mu \nu \rho} F_{\mu \nu} +  2e Im( \phi  \overline{D^{\rho} \phi})  = 0\\
\label{23}
&D_{\mu} D^{\mu} \phi  + U_{\overline{\phi}}(|\phi|^2,N) = 0 \\
\label{24}
& \partial_{\mu} \partial^{\mu} N + U_N (|\phi|^2,N) = 0 \, ,
\end{align}
where
\begin{align}
\label{25}
U_{\overline{\phi}}(|\phi|^2,N) &= (e|\phi|^2+\kappa N)\phi + e^2(N+\frac{ev^2}{\kappa})^2 \phi \\
\label{25''}
 U_N(|\phi|^2,N) &= \kappa (e |\phi|^2 + \kappa N) +2e^2(N+\frac{ev^2}{\kappa}) |\phi|^2 
\end{align}
with initial conditions 
\begin{align}
\label{25'}
A_{\mu}(0) = a_{\mu 0} \, , \, \partial_tA_{\mu}(0) = a_{\mu 1} \, , \, \phi_0 = \phi_0 \, , \, \partial_t \phi(0) = \phi_1 \, , \, N(0)=N_0 \, , \, \partial_t N(0) = N_1 \, ,
\end{align}
 which fulfills the Lorenz condition $ \partial^{\mu} A_{\mu} = 0 \, . $
\end{theorem} 
\begin{Cor}
\label{Corollary1}
Let $s> \frac{3}{4}$ , $T>0$ and
$\phi_0,a_{\mu 0},n_0 \in H^s$ , $a_{\mu 1},n_1 \in H^{s-1}$ satisfying (\ref{GL}),(\ref{21}). Then the solution of (\ref{22}),(\ref{23}),(\ref{24}) with initial conditions (\ref{25'}) under the Lorenz condition $\partial^{\mu} A_{\mu}=0$ is (unconditionally) unique in the space $\phi,A_{\mu},N \in C^0([0,T],H^s) \cap C^1([0,T],H^{s-1})$. Combined with the existence result of Yuan \cite{Y} we obtain local well-posedness and in energy space and above ($s \ge 1$) global well-posedness.
\end{Cor}
 
Fundamental for the proof of our theorem are the following bilinear estimates in wave-Sobolev spaces which were proven by d'Ancona, Foschi and Selberg in the two-dimensional case $n=2$ in \cite{AFS} in a more general form which include many limit cases which we do not need.
\begin{theorem}
\label{Theorem3}
Let $n=2$. The estimate
$$\|uv\|_{X^{-s_0,-b_0}} \lesssim \|u\|_{X^{s_1,b_1}} \|v\|_{X^{s_2,b_2}} $$ 
holds, provided the following conditions are satisfied:
$$b_0 + b_1 + b_2 > \frac{1}{2} \, ,\quad
 b_0 + b_1 > 0 \, ,\quad
 b_0 + b_2 > 0  \, ,\quad
 b_1 + b_2 > 0  \, ,$$
\begin{align*}
&s_0+s_1+s_2 > \frac{3}{2} -(b_0+b_1+b_2) \\
\nonumber
&s_0+s_1+s_2 > 1 -\min(b_0+b_1,b_0+b_2,b_1+b_2) \\
\nonumber
&s_0+s_1+s_2 > \frac{1}{2} - \min(b_0,b_1,b_2) \\
\nonumber
&s_0+s_1+s_2 > \frac{3}{4} \\
\end{align*}
\begin{align*}
 &(s_0 + b_0) +2s_1 + 2s_2 > 1 \\
\nonumber
&2s_0+(s_1+b_1)+2s_2 > 1 \\
\nonumber
&2s_0+2s_1+(s_2+b_2) > 1 \\
\nonumber
&s_1 + s_2 \ge \max(0,-b_0) \\
\nonumber
&s_0 + s_2 \ge \max(0,-b_1) \\
\nonumber
&s_0 + s_1 \ge \max(0,-b_2)   \, .
\end{align*}
\end{theorem}
{\bf Acknowledgment:} The author thanks the referee for valuable comments and suggestions which helped to improve the paper.

\section{Proof of Theorem \ref{Theorem1}}
An application of the contraction mapping is by well-known arguments reduced to suitable multilinear estimates of the right hand sides of (\ref{2.1}),(\ref{2.2}) and (\ref{2.3}). 
For (\ref{1.1}) e.g. we make use of the following well-known estimate:
\begin{align*}
\|A_{0,\pm}\|_{X^{l,b}_{\pm}[0,T]} &\lesssim \|A_{0,\pm}(0)\|_{H^l} + T^{\epsilon} \| \langle \nabla \rangle^{-1} (R.H.S. \, of \, (\ref{1.1})) \|_{X^{l,b-1+\epsilon}_{\pm}[0,T]} \\
&\lesssim \|A_{0,\pm}(0)\|_{H^l} + T^{\epsilon} \| R.H.S.\, of \,(\ref{1.1}) \|_{X^{l-1,b-1+\epsilon}_{\pm}[0,T]} \, ,
\end{align*} 
which holds for $l\in{\mathbb R}$ , $0<T\le 1$ , $\frac{1}{2} < b \le 1$ , $0 < \epsilon \le 1-b$ .

The linear terms are easily treated and therefore omitted here.

We now consider the right hand side of (\ref{1.1}):
$$ -2e Im(\phi \overline{D_0 \phi}) = -2e Im(\phi_++\phi_-)(-i)\langle \nabla \rangle(\overline{\phi}_+ - \overline{\phi}_-)) -2e^2 A_0 |\phi|^2 \, . $$
In the case $\frac{1}{2} < s < \frac{5}{8}$ we need the estimate
$$ \|\phi_{\pm 1} \langle \nabla \rangle \overline{\phi}_{\pm 2} \|_{X^{2s-\frac{7}{4}-,2s-\frac{5}{4}-}_{\pm}} \lesssim \|\phi_{\pm 1} \|_{X^{s,\frac{1}{2}+}_{\pm 1}}  \|\langle \nabla \rangle \phi_{\pm 2} \|_{X^{s-1,\frac{1}{2}+}_{\pm 2}} \, , $$
where here and in the following $\pm$ and $\pm_j$ denote independent signs.
This follows from
$$ \|uv\|_{X^{2s-\frac{7}{4}-,2s-\frac{5}{4}-}} \lesssim \|u\|_{X^{s,\frac{1}{2}+}}     \|v\|_{X^{s-1,\frac{1}{2}+}} \, , $$
which holds by Theorem \ref{Theorem3} with $s_0 = -2s+\frac{7}{4}+$ , $b_0 = \frac{5}{4}-2s+$ , $s_1=s$ , $s_2=s-1$,  $b_1=b_2= \frac{1}{2}+$ , so that $s_0+s_1+s_2 > \frac{3}{4}$ , $s_1+s_2 > 0$ for $ s > \frac{1}{2}$ and $b_0 > 0$ for $s \le \frac{5}{8}$ . \\
In the case $\frac{5}{8} < s < 1$ we have to show
$$ \|uv\|_{X^{2s-\frac{7}{4}-,0}} \lesssim \|u\|_{X^{s,\frac{1}{2}+}} \|v\|_{X^{s-1,\frac{1}{2}+}} \, ,$$
which follows similarly from Theorem \ref{Theorem3} with the same choice of the parameters as before except $b_0 =0$. This implies $s_0+b_0+2s_1+2s_2 > 1$ for $s > \frac{5}{8}$ .\\
The cubic term $2e^2 A_0 |\phi|^2$ is handled as follows. In the case $\frac{1}{2} < s \le \frac{5}{8}$ we obtain
$$\|uvw\|_{X^{2s-\frac{7}{4}-,2s-\frac{5}{4}-}} \lesssim \|u\|_{X^{2s-\frac{3}{4}-,2s-\frac{1}{4}--}} \|vw\|_{X^{0,0}}$$
and in the case $ \frac{5}{8} < s < 1$
$$\|uvw\|_{X^{2s-\frac{7}{4}-,0}} \lesssim \|u\|_{X^{2s-\frac{3}{4}-,1-}} \|vw\|_{X^{0++,0}}$$
by Theorem \ref{Theorem3}. Combining this with the estimate
$$ \|vw\|_{X^{0++,0}} \lesssim \|v\|_{X^{s,\frac{1}{2}+}} \|w\|_{X^{s,\frac{1}{2}+}} \, ,$$ 
which easily follows by Sobolev's embedding for $s > \frac{1}{2}$ we obtain the desired estimate.

The right hand side of (\ref{1.2}) can be handled in the same way.

It remains to consider the right hand side of (\ref{1.3}). We start with the most interesting quadratic term, where the null conditions come into play, namely
$ 2ie A_{\mu} \partial^{\mu} \phi  \, . $
Defining the modified Riesz transforms $R_j := \langle \nabla \rangle^{-1} \partial_j$ and splitting $A_j$ into divergence-free and curl-free parts and a smooth remainder we obtain
$$A_j = A_j^{df} + A_j^{cf} + \langle \nabla \rangle^{-2} A_j \, , $$
where
\begin{align*}
& A_1^{df} = R_2(R_1 A_2-R_2 A_1) & & A_2^{df} = R_1(R_2A_1-R_1 A_2) \\
& A_1^{cf} = -R_1(R_1A_1 + R_2 A_2) & &A_2^{cf} = -R_2(R_1A_1+R_2A_2) \, . 
\end{align*}
Now we have
\begin{equation}
\label{***}
 A_{\mu} \partial^{\mu} \phi = (A_0 \partial_t \phi - A_j^{cf} \partial^j \phi) - A_j^{df} \partial^j \phi + \langle \nabla \rangle^{-2} A_j \partial^j \phi \, . 
 \end{equation}
The first two terms on the right hand side are of null form type, where for the first term we have to use the Lorenz condition (these arguments go back to Selberg-Tesfahun). \\
We calculate
\begin{align*}
A_j^{df} \partial^j \phi & = R_2(R_1A_2 - R_2 A_1) R_1 \langle \nabla \rangle \phi -R_1(R_1A_2-R_2A_1) R_2 \langle \nabla \rangle \phi \\
& = \sum_{\pm_1,\pm_2} Q^{12}_{\pm 1,\pm 2} ((R_1 A_{2,\pm 1} - R_2 A_{1,\pm 1}), \langle \nabla \rangle \phi_{\pm 2}) \, ,
\end{align*}
where
$$ Q^{12}_{\pm 1,\pm 2} (u,v) := R_2 (\pm_1 u) R_1 (\pm_2 v) - R_1(\pm_1 u) R_2 (\pm_2 v) $$
is the standard null form with Fourier symbol
$$ \sigma_1(\pm_1 \xi, \pm_2 \eta) := \frac{(\pm_1 \xi_2)(\pm_2 \eta_1) - (\pm_1 \xi_1)(\pm_2 \eta_2)}{\langle \xi \rangle \langle \eta \rangle} \, , $$
which can be estimated by
\begin{equation}
\label{sigma1}
|\sigma_1(\pm_1 \xi, \pm_2 \eta)| \lesssim \frac{|\xi||\eta|}{\langle \xi \rangle \langle \eta \rangle}  \Theta(\pm_1 \xi,\pm_2\eta) \le \Theta(\pm_1 \xi,\pm_2\eta) \, , 
\end{equation}
where $\Theta(\xi,\eta)$ denotes the angle between two vectors $\xi,\eta \in {\mathbb R}^2$.\\
Next the Lorenz condition gives
$$R_1 A_1 + R_2 A_2 = \langle \nabla \rangle^{-1}(\partial_1 A_1 + \partial_2 A_2) =
 \langle \nabla \rangle^{-1}\partial_t A_0 \, , $$
so that
$$A_j^{cf} = -R_j(R_1 A_1 + R_2 A_2) = - \langle \nabla \rangle^{-1}R_j \partial_t A_0 = -iR_j(A_{0,+}-A_{0,-}) \, . $$
Thus
\begin{align*}
A_0 \partial_t \phi - A_j^{cf}\partial^j \phi & = (A_{0,+}+A_{0,-})i \langle \nabla \rangle (\phi_+ - \phi_-) -i R_j(A_{0,+}-A_{0,-})\partial^j(\phi_+ + \phi_-) \\
& = i \sum_{\pm_1,\pm_2}(A_{0,\pm 1} \langle \nabla \rangle(\pm_2 \phi_{\pm 2}) - 
R_j(\pm_1 A_{0,\pm 1}) \partial^j \phi_{\pm 2}) \\ 
& = i  \sum_{\pm_1,\pm_2} (A_{0,\pm 1} (\pm_2 \langle \nabla \rangle \phi_{\pm 2})-R_j (\pm_1 A_{0,\pm 1}) R^j \langle \nabla \rangle \phi_{\pm 2})) \\
& =i \sum_{\pm_1,\pm_2} \tilde{Q}_{\pm 1,\pm 2} (A_{0,\pm 1} , \langle \nabla \rangle \phi_{\pm 2}) \, ,
\end{align*}
where
$$ \tilde{Q}_{\pm 1,\pm 2} (u,v) := u (\pm_2 v) - (\pm_1 R_j u)R^j v $$
has Fourier symbol
$$ \sigma_2(\pm_1 \xi,\pm_2 \eta) := (\pm_2 1) - \frac{(\pm_1 \xi) \cdot \eta}{\langle \xi \rangle \langle \eta \rangle} = \frac{ \pm_2 \langle \xi \rangle \langle \eta \rangle - (\pm_1 \xi) \cdot \eta}{\langle \xi \rangle \langle \eta \rangle} \, . $$ 
Now
\begin{align*}
& |\pm_2 \langle \xi \rangle \langle \eta \rangle - (\pm_1 \xi) \cdot \eta | = \langle \eta \rangle \big| \langle \xi \rangle - \frac{(\pm_1 \xi) \cdot (\pm_2 \eta)}{\langle \eta \rangle} \big| \\
& \le |\eta|\big(\big| \langle \xi \rangle - \frac{(\pm_1 \xi) \cdot (\pm_2 \eta)}{|\eta|} \big| + \big| (\pm_1 \xi) \cdot (\pm_2 \eta) \big( \frac{1}{|\eta|} - \frac{1}{\langle \eta \rangle} \big) \big| \big)\\
& \le |\eta|\big( (\langle \xi \rangle - |\xi|) + |\xi|\big( 1 - \frac{(\pm_1 \xi) \cdot (\pm_2 \eta)}{|\eta| |\xi|} \big) \big) + \big|(\pm_1 \xi)(\pm_2 \eta) \big(\frac{1}{|\eta|} - \frac{1}{\langle \eta \rangle} \big) \big| \\
& \lesssim |\eta| \big( 1 + |\xi|  \Theta(\pm_1 \xi,\pm_2 \eta) + \frac{|\xi|}{\langle \eta \rangle} (\langle \eta \rangle - |\eta|) \big) \\
& \lesssim \langle \eta \rangle + \langle \eta \rangle \langle \xi \rangle \Theta(\pm_1 \xi,\pm_2 \eta) + \langle \xi \rangle \, ,
\end{align*}
so that 
\begin{equation}
\label{sigma2}
 |\sigma_2(\pm_1 \xi,\pm_2 \eta)| \lesssim \Theta(\pm_1 \xi, \pm_2 \eta) + \frac{1}{\langle \xi \rangle} + \frac{1}{\langle \eta \rangle} \, . 
\end{equation}
Our aim is to prove the following estimate in the case $\frac{1}{2} < s \le \frac{5}{8}$ :
\begin{equation}
\label{A}
\|A_{\mu,\pm 1} \partial^{\mu} \phi_{\pm 2}\|_{X^{s-1,-\frac{1}{2}++}_{\pm}} \lesssim \sum_{\mu} \|A_{\mu,\pm 1} \|_{X^{2s-\frac{3}{4}-,2s-\frac{1}{4}--}_{\pm 1}} \| \langle \nabla \rangle \phi_{\pm 1} \|_{X^{s-1,\frac{1}{2}+}_{\pm 2}}
\end{equation}
and in the case $\frac{5}{8} < s < 1$:
\begin{equation}
\label{A'}
\|A_{\mu,\pm 1} \partial^{\mu} \phi_{\pm 2}\|_{X^{s-1,-\frac{1}{2}++}_{\pm}} \lesssim \sum_{\mu} \|A_{\mu,\pm 1} \|_{X^{2s-\frac{3}{4}-,1-}_{\pm 1}} \| \langle \nabla \rangle \phi_{\pm 1} \|_{X^{s-1,\frac{1}{2}+}_{\pm 2}}
\end{equation} 
We first estimate the last term on the right hand side of (\ref{***}) in the whole range $\frac{1}{2} < s < 1$ . We have the sufficient estimate
$$ \| \langle \nabla \rangle^{-2} A_{j,\pm 1} \partial^j \phi_{\pm 2}\|_{X^{s-1,-\frac{1}{2}++}_{\pm}} \lesssim \|A_{j,\pm 1} \|_{X^{2s-\frac{3}{4}-,\frac{3}{4}--}_{\pm 1}} \|\partial^j \phi_{\pm 2}\|_{X^{s-1,\frac{1}{2}+}_{\pm 2}} \, , $$
which follows from
$$\|uv\|_{X^{s-1,-\frac{1}{2}++}} \lesssim \|u\|_{X^{2s+\frac{5}{4}-,\frac{3}{4}--}} \|v\|_{X^{s-1,\frac{1}{2}+}} $$
by an application of Sobolev's embedding using $(1-s)+(2s+\frac{5}{4})+(s-1) > 1$. 
The claimed estimate for the first two terms on the right hand side of (\ref{***})
reduces to
\begin{align}
\nonumber
 \int \frac{\widehat{u}(\xi,\tau)}{\langle -\tau \pm_1 
|\xi|\rangle^{2s-\frac{1}{4}--} \langle \xi \rangle^{2s-\frac{3}{4}-}} 
\frac{\widehat{v}(\eta,\lambda)}{\langle -\lambda \pm_2 |\eta|\rangle^{\frac{1}{2}+} \langle \eta \rangle^{s-1}} \\
\label{I}
\frac{\widehat{w}(\xi + \eta,\tau + \lambda)}{\langle |\lambda + \tau| - |\xi + \eta|  \rangle^{\frac{1}{2}--} \langle \xi + \eta \rangle^{1-s}} |\sigma_j(\pm_1 \xi,\pm_2 \eta)|  d\xi d\tau d\eta d\lambda
 \\ \nonumber
\lesssim \|u\|_{L^2_{xt}}\|v\|_{L^2_{xt}}\|w\|_{L^2_{xt}} \, ,
\end{align}
for $j=1,2$ ,
where we may assume that the Fourier transforms are nonnegative. In order to estimate the symbols using (\ref{sigma1}) and (\ref{sigma2}) we use the following 
\begin{lemma} (\cite{ST}, Lemma 4.3)
\begin{equation}
\label{angle}
\Theta(\pm_1 \xi,\pm_2 \eta) \lesssim \Big(\frac{\langle -\tau \pm_1 |\xi| \rangle + \langle -\lambda \pm_2 |\eta| \rangle}{\min(\langle \xi\rangle,\langle \eta \rangle)}\Big)^{\frac{1}{2}} + \Big(\frac{ \langle |\lambda + \tau|-|\eta +\xi| \rangle}{\min(\langle \xi\rangle,\langle \eta \rangle)}\Big)^{\frac{1}{2}--}
\end{equation}
$\forall \, \xi,\eta \in{\mathbb R}^2 \, , \, \lambda,\tau \in {\mathbb R} \, .$
\end{lemma}
\noindent Thus (\ref{I}) reduces to the following estimates
\begin{align}
\label{1}
\|uv\|_{X^{s-1,0}} &\lesssim \|u\|_{X^{2s-\frac{1}{4}-,2s-\frac{1}{4}--}} \|v\|_{X^{s-1,\frac{1}{2}+}} \\
\label{2}
\|uv\|_{X^{s-1,0}} &\lesssim \|u\|_{X^{2s-\frac{3}{4}-,2s-\frac{1}{4}--}} \|v\|_{X^{s-\frac{1}{2}--,\frac{1}{2}+}} \\
\label{3}
\|uv\|_{X^{s-1,-\frac{1}{2}++}} &\lesssim \|u\|_{X^{2s-\frac{1}{4}-,2s-\frac{3}{4}--}} \|v\|_{X^{s-1,\frac{1}{2}+}} \\
\label{4}
\|uv\|_{X^{s-1,-\frac{1}{2}++}} &\lesssim \|u\|_{X^{2s-\frac{3}{4}-,2s-\frac{3}{4}--}} \|v\|_{X^{s-\frac{1}{2},\frac{1}{2}+}} \\
\label{5}
\|uv\|_{X^{s-1,-\frac{1}{2}++}} &\lesssim \|u\|_{X^{2s-\frac{1}{4}-,2s-\frac{1}{4}--}} \|v\|_{X^{s-1,0+}} \\
\label{6}
\|uv\|_{X^{s-1,-\frac{1}{2}++}} &\lesssim \|u\|_{X^{2s-\frac{3}{4}-,2s-\frac{1}{4}--}} \|v\|_{X^{s-\frac{1}{2},0+}} \\
\label{7}
\|uv\|_{X^{s-1,-\frac{1}{2}++}} &\lesssim \|u\|_{X^{2s+\frac{1}{4}-,\frac{3}{4}--}} \|v\|_{X^{s-1,\frac{1}{2}+}} \\
\label{8}
\|uv\|_{X^{s-1,-\frac{1}{2}++}} &\lesssim \|u\|_{X^{2s-\frac{3}{4}-,\frac{3}{4}--}} \|v\|_{X^{s,\frac{1}{2}+}} \, ,
\end{align} 
where the last two estimates take account of the last two terms in (\ref{sigma2}) and suffice for the whole range $\frac{1}{2} < s < 1$ . In the case $\frac{5}{8} < s < 1$ the terms of the form $\|u\|_{X^{k,2s-\frac{1}{4}--}}$ have to be replaced by $\|u\|_{X^{k,1-}}$  and $\|u\|_{X^{k,2s-\frac{3}{4}--}}$ by $\|u\|_{X^{k,-\frac{1}{2}-}}$ . All these estimates follow from Theorem \ref{Theorem3}. (\ref{7}) and (\ref{8}) are easy consequences of the Sobolev embedding theorem using $1-s+2s+\frac{1}{2}+s-1 > \frac{5}{4} > 1$. \\
We consider only (\ref{1}). The parameters in Theorem \ref{Theorem3} are: \\
$s_0=1-s$ , $b_0 =0$ , $ s_1=2s-\frac{1}{4}-$ , $b_1=2s-\frac{1}{4}--$ , $s_2=s-1$ , $b_2 = \frac{1}{2}+$ \\
so that the condition $s_0+s_1+s_2=2s-\frac{1}{4}- > \frac{3}{4}$ is fulfilled for $s>\frac{1}{2}$ and also the condition $s_0+b_0+2s_1+2s_2 = 5s - \frac{3}{2} - > 1$.
The other conditions to be fulfilled in Theorem \ref{Theorem3} can be shown to be less critical.\\
Similar elementary calculations show that (\ref{2}) - (\ref{6}) are satisfied. The corresponding estimates in the range $\frac{5}{8} < s < 1$ follow similarly. This completes the proof of (\ref{A}) and (\ref{A'}).

The remaining terms on the right hand side of (\ref{1.3}) can be handled easier in the following way for the whole range $ \frac{1}{2} < s < 1$.

Concerning $A^2_{\mu} \phi$ we show
$$ \|uvw\|_{X^{s-1,-\frac{1}{2}++}} \lesssim \|uv\|_{X^{s-\frac{3}{4},0}} \|w\|_{X^{s,\frac{1}{2}+}} \lesssim \|u\|_{X^{2s-\frac{3}{4}-,\frac{3}{4}-}} \|v\|_{X^{2s-\frac{3}{4}-,\frac{3}{4}-}} \|w\|_{X^{s,\frac{1}{2}}} $$
by two applications of Theorem \ref{Theorem3}.

For the terms $|\phi|^2 \phi$ and $N^2 \phi$ we obtain similarly
$$ \|uvw\|_{X^{s-1,-\frac{1}{2}++}} \lesssim \|uv\|_{X^{0+,0}} \|w\|_{X^{s,\frac{1}{2}+}} \le \|u\|_{X^{\frac{1}{2},\frac{1}{2}+}} \|v\|_{X^{\frac{1}{2},\frac{1}{2}+}} \|w\|_{X^{s,\frac{1}{2}+}} \, . $$

The term $N\phi$  can be handled even more easily.

Finally we have to consider the terms on the right hand side of (\ref{1.4}). The term  $N|\phi|^2$ (and similarly $|\phi|^2$) is treated by Theorem \ref{Theorem3} and Sobolev as follows:
$$\|uvw\|_{X^{-\frac{1}{2},-\frac{1}{2}++}} \lesssim \|u\|_{X^{\frac{1}{2},\frac{1}{2}+}} \|vw\|_{X^{0+,0}} \lesssim \|u\|_{X^{\frac{1}{2},\frac{1}{2}+}} \|v\|_{X^{s,\frac{1}{2}+}} \|w\|_{X^{s,\frac{1}{2}+}} \, .$$

The proof of Theorem \ref{Theorem1} is now complete.

\section{Proof of Theorem \ref{Theorem2} and Corollary \ref{Corollary1}}
\begin{proof}[Proof of Theorem \ref{Theorem2}]
In a first step we show that (\ref{1.1*}) and the Lorenz condition $\partial^{\mu} A_{\mu} =0$ are satisfied , because they hold initially, i.e. ,  (\ref{GL}) and (\ref{21}) are assumed.  Let
$$ W:= \partial_t A_0 - \partial_1 A_1 - \partial_2 A_2 \, , \, V:= \Delta A_0 - \kappa F_{12} -2e Im(\phi \overline{D^0 \phi}) - \partial_t(\partial_1 A_1 + \partial_2 A_2) \, . $$
With this notation we know that $W(0)=V(0)=0$ and want to show $W(t)=V(t)=0$ $\forall t \in [0,T]$. We claim that (W,V) is a solution of the system
\begin{equation} \label{9}
 \partial_t W= V \quad , \quad \partial_t V = \Delta W + 2e^2 |\phi|^2 W\, . 
 \end{equation}
We easily calculate using (\ref{1.1})
\begin{align}
\partial_t W &= \partial_t^2 A_0-\partial_t(\partial_1 A_1+ \partial_2 A_2) \\
\nonumber
& = \Delta A_0 - \kappa F_{12} -2e Im(\phi \overline{D^0 \phi}) - \partial_t(\partial_1 A_1 + \partial_2 A_2) = V \, .
\end{align}
Moreover, by (\ref{1.2})
\begin{align*}
\Delta W & = \partial_t A_0 -\partial_1 \Delta A_1 -\partial_2 \delta A_2 \\
& = \partial_t \Delta A_0 - \partial_1(\kappa F_{02} +2e Im(\phi \overline{D^1 \phi}) - \partial_1 \partial_t^2 A_1 \\
& \hspace{4.5em} - \partial_2(-\kappa F_{01} +2e Im(\phi \overline{D^2 \phi})) - \partial_2 \partial_t^2 A_2 \\
& = \partial_t \Delta A_0 - \partial_t( \kappa F_{12}) - 2e(\partial_1 Im(\phi \overline{D^1 \phi}) + \partial_2 Im(\phi \overline{D^2 \phi})) - \partial_t^2(\partial_1 A_1 + \partial_2 A_2) \, .
\end{align*}
Now we calculate
\begin{align*}
& \partial_1 Im(\phi \overline{D^1 \phi}) + \partial_2 Im(\phi \overline{D^2 \phi}) \\
& = Im(\phi ( \overline{\partial_1 D^1 \phi} + \overline{\partial_2 D^2 \phi})) + Im(\partial_1 \phi \overline{D^1 \phi} + \partial_2 \phi \overline{D^2 \phi}) \\
& = Im(\phi(\overline{D_1^2 \phi}) + \overline{D_2^2 \phi})) + Im(\phi(-ie A_1\overline{D^1 \phi} -ie A_2 \overline{D^2 \phi})) \\
& \quad + Im(D_1 \phi \overline{D_1 \phi} + ie A_1 \phi \overline{D^1 \phi}) + Im(D_2 \phi \overline{D_2 \phi} + ie A_2 \phi \overline{D_2 \phi)}) \\
& = Im(\phi(\overline{D_1^2 \phi} + \overline{D_2^2 \phi}))
\end{align*}
and
\begin{align*}
& D_0^2 \phi - D_1^2 \phi - D_2^2 \phi \\
& = ((\partial_0^2 - \partial_1^2 -\partial_2^2) \phi - 2ie A_0 \partial_0 \phi + 2ie A_j \partial^j \phi + e^2  A^j A_j \phi - e^2 A_0^2 \phi) \\
& \quad+ ie(\partial_1 A_1 + \partial_2 A_2 - \partial_0 A_0) \phi \\
& = ie(\partial_1 A_1 + \partial_2 A_2 - \partial_0 A_0) \phi - U_{\overline{\phi}}(|\phi|^2,N) = -ie W \phi - U_{\overline{\phi}}(|\phi|^2,N) \, ,
\end{align*}
where we used (\ref{1.3}) in the last line. This implies
\begin{align*}
 \partial_1 Im(\phi \overline{D^1 \phi}) + \partial_2 Im(\phi \overline{D^2 \phi})  &= Im(\phi \overline{D_0^2\phi}) + e Im(i \phi W \overline{\phi}) - Im(\phi \overline{U_{\overline{\phi}}(|\phi|^2,N)}) \\
& = Im(\phi \overline{D_0^2  \phi}) + e W |\phi|^2
\end{align*}
using $Im(\phi \overline{U(|\phi|^2,N)}) = 0$ .
Furthermore
\begin{align*}
& Im(\phi \overline{D_0^2 \phi}) = Im(\phi (\overline{\partial_t - ie A_0)D_0 \phi}) \\
& = \partial_t Im(\phi \overline{D_0 \phi}) - Im(\partial_t \phi \overline{D_0 \phi}) + e Im(\phi i A_0 \overline{\partial_t \phi}) - e^2 Im(\phi A_0^2 \overline{\phi}) \\
& = \partial_t Im(\phi \overline{D_0 \phi}) - Im(\partial_t \phi \overline{\partial_t \phi}) - e Im(\partial_t \phi i A_0 \overline{\phi}) - e Im(\phi (-i) A_0 \overline{\partial_t \phi}) \\
& = \partial_t Im(\phi \overline{D_0 \phi}) \, .
\end{align*}
Collecting all these calculations we finally arrive at
$$ \Delta W = \partial_t V - 2e^2 |\phi|^2 W \, , $$
so that (\ref{9}) is proven and
\begin{equation}
\label{10}
(\partial_t^2 - \Delta)W = 2e^2 |\phi|^2 W \, .
\end{equation}
We remind that $W(0) = 0 \, , \, (\partial_t W)(0)=V(0) = 0$. For sufficiently regular $\phi$ this Cauchy problem for a linear wave equation is uniquely solvable by $W(t) \equiv 0$ . Thus by (\ref{9}) also $V(t) \equiv 0$, especially the Lorenz condition is satisfied. But the solutions obtained in Theorem \ref{Theorem1} are continously depending on the data and also persistence of higher regularity holds. Thus by regularization of the data we may assume here $\phi,\partial_t \phi \in C_0^{\infty}([0,t] \times {\mathbb R}^2)$ and $A_{\mu},\partial A_{\mu} \in C^{\infty}([0,T] \times {\mathbb R}^2)$ , which justifies our calculations.

Under the Lorenz condition $W \equiv 0$ the equations (\ref{1.1})and (\ref{1.2}) exactly reduce to (\ref{1.1*}),(\ref{1.2*}) and (\ref{1.3*}), which means that (\ref{2.1}) is fulfilled. Under the Lorenz condition we also have that (\ref{2.2}) and (\ref{2.3}) are equivalent to (\ref{1.3}) and (\ref{1.4}), respectively.

Summarizing we have shown that the solution of Theorem \ref{Theorem1} is a solution of the Cauchy problem for (\ref{2.1}),(\ref{2.2}),(\ref{2.3}) satisfying the Lorenz condition. 
As remarked earlier already the reverse is also true, so that uniqueness holds. 
\end{proof}
\begin{proof}[Proof of Corollary \ref{Corollary1}]
It suffices to show that any solution $\phi,A_{\mu},N \in C^0([0,T],H^s)$ $ \cap C^1([0,T],H^{s-1})$ belongs to $\phi,A_{\mu},N \in X^{\frac{1}{2}+,1}[0,T]$ , where uniqueness holds by Theorem \ref{Theorem2}. This follows if the right hand sides of (\ref{1.1'}), (\ref{1.2'}), (\ref{1.3'}) and (\ref{1.4'}) belong to $L^2([0,T],H^{\frac{1}{2}+})$. The quadratic terms are all treated similarly, e.g. $\langle \nabla \rangle^{-1} (A_{\mu} \partial_{\mu} \phi)$ is estimated by the Sobolev multiplication rule as follows:
$$\|A_{\mu} \partial_{\mu} \phi \|_{L_t^2([0,T],H^{-\frac{1}{2}+}_x)} \lesssim T^{\frac{1}{2}} \|A_{\mu}\|_{L^{\infty}_t([0,T],H^s)} \|\partial_{\mu} \phi\|_{L^{\infty}_t([0,T],H^{s-1})} < \infty \, . $$
The cubic terms are even easier to handle, e.g.
$$\|A_{\mu}^2 \phi \|_{L^2_t([0,T],H^{-\frac{1}{2}+}_x)} \lesssim T^{\frac{1}{2}} \|A_{\mu}\|^2_{L^{\infty}_t([0,T],L_x^4)} \|\phi\|_{L^{\infty}_t([0,T],H^s)} <\infty \, .$$
\end{proof}

\end{document}